\documentclass{amsart}
\usepackage[utf8]{inputenc}
\usepackage{ifxetex,ifluatex}

\usepackage{colonequals}
\usepackage{amsaddr}
\usepackage{float}
\usepackage{tikz-cd}
\usepackage{amsmath}%
\usepackage{amsfonts}%
\usepackage{amsthm}
\usepackage{amssymb}%
\usepackage{graphicx}
\usepackage{enumerate}
\usepackage{hyperref}
\usepackage[all,cmtip]{xy}
\usepackage{afterpage}
\usepackage{geometry}
 \usepackage{comment}
\usepackage{cleveref}
\crefformat{section}{§#2#1#3} 
\usepackage{csquotes}
\usepackage{stmaryrd}
\usepackage{lscape} 
\usepackage{pdflscape} 
\usepackage{rotating}
\usepackage{tabularx}
\usepackage{xcolor}

\theoremstyle{plain}
\newtheorem{theorem}{Theorem}[section]
\newtheorem{lemma}[theorem]{Lemma}
\newtheorem{corollary}[theorem]{Corollary}

\newtheorem{question}[theorem]{Question}


\theoremstyle{definition}

\theoremstyle{remark}

\newcommand{\Z}{\mathbb{Z}}
\newcommand{\Q}{\mathbb{Q}}
\renewcommand{\P}{\mathcal{P}}
\newcommand{\QQ}{\mathcal{Q}}

\usepackage{xcolor}

\usepackage{xpatch}

\makeatletter
\xpatchcmd{%
\@maketitle}{%
\ifx\@empty\authors \else \@setauthors \fi
}{%
  \ifx\@empty\authors \else \@setauthors \fi
  \ifx\@empty\addresses \else\@setaddresses\fi
}{\typeout{Patch successful}}{\typeout{Patch failed}}
\makeatother

\begin{document}
\title[Reverse Engineered Diophantine Equation]{Reverse Engineered Diophantine Equations}

\author{Stevan Gajovi\'{c}}

\maketitle
\begin{abstract}
We answer a question of Samir Siksek, asked at the open problems session of the conference ``Rational Points 2022'', which, in a broader sense, can be viewed as a reverse engineering of Diophantine equations. For any finite set $S$ of perfect integer powers, using Mih\u{a}ilescu's theorem, we construct a polynomial $f_S\in \Z[x]$ such that the set $f_S(\Z)$ contains a perfect integer power if and only if it belongs to $S$. We first discuss the easier case where we restrict to all powers with the same exponent. In this case, the constructed polynomials are inspired by Runge's method and Fermat's Last Theorem. Therefore we can construct a polynomial-exponential Diophantine equation whose solutions are described in advance. 
\end{abstract}


\section{Introduction}
Diophantine problems look innocent but often are tricky. They can be formulated using very basic mathematics, but it turns out that solving many of them took centuries of serious work by numerous mathematicians. Some of them are still not solved; they are famous conjectures in mathematics and, more precisely, in number theory. 

The most famous Diophantine equations include some relations with perfect integer powers, and we briefly survey them in \S\ref{subsec:FamousDE}. We will use some of these equations in
the main results of this paper in \S\ref{subsec:modular-approach} and \S\ref{sec:General-case}.

When we also include rational solutions, to simplify the notation, we will talk about Diophantine problems. We also mention in \S\ref{subsec:RPoncurves} one instance of famous Diophantine problems, classifying the set of rational points $C(\Q)$ on curves $C$ defined over $\Q$. We briefly use some of the results and techniques mentioned there in \S\ref{subsec:modular-approach}.

In \S\ref{sec:FamousDE}, we note that Diophantine problems are an active and very challenging area of mathematics. On the other hand, there are equations that are trivial to solve, such as linear Diophantine equations, or some equations that have obvious ways to be solved, e.g., by considering them modulo $n$ for some $n\in \Z_{>1}$, or, more generally, using local obstructions.

Therefore, it is an interesting problem to look into the middle case - can we create similar Diophantine equations with a prescribed set of solutions so that we can solve them more easily, but not that obviously? The main task of this article is to take a different perspective - and construct Diophantine equations of a specific shape with a described set of solutions. That was precisely the question Samir Siksek asked during the conference ``Rational Points 2022''. Namely, if we consider a Diophantine equation of the type
\[
f(x)=y^n,
\]
for a given $f\in \Z[x]$, with unknowns $x,y,n\in \Z$ and $n\geq2$, can it happen that a triple $(x,y,n)$ is a solution if and only if $y^n$ belongs to a finite set of integer powers given in advance?\\

\subsection{Main results}\label{subsec:main-results}
We now reformulate Siksek's question and make it more precise. Let 
\[\P=\{a^m\colon a\in \Z, m\geq 2\}\subseteq \Z\]
denote the set of all perfect integer powers.

\begin{question}[Siksek]\label{q?-1}
Let $S\subseteq \P$ be a finite set of perfect powers. Is there a polynomial $f_S\in \Z[x]$ such that $f_S(\Z)\cap \P = S$?
\end{question}
Here we give an affirmative answer to Question \ref{q?-1} by constructing such a polynomial in \S\ref{sec:General-case}, see Theorem \ref{thm:Second-general}. \\

One could ask an easier question. Let $m\geq 2$ be a fixed integer. Denote by $\P_m=\{a^m\colon a\in \Z\}$ the set of all perfect $m$th powers of integers.

\begin{question}[Siksek]\label{q?-2}
Let $S\subseteq \P_m$ be a finite set of $m$th powers. Is there a polynomial $f_S\in \Z[x]$ such that $f_S(\Z)\cap \P_m = S$?
\end{question}

We recall that in Question \ref{q?-2}, we fix $m\in\Z_{\geq 2}$. Curves of the shape $C:y^m=f(x)$ are called {\it superelliptic curves}. We can also rephrase Question \ref{q?-2} as a task to construct superelliptic curves whose integral points have prescribed $y$-coordinates in advance.  \\

We explain two different methods to construct a polynomial as asked for in Question \ref{q?-2} (one method only works for $m\geq 3$) in \S\ref{subsec:Runge-approach} and \S\ref{subsec:modular-approach}, see Theorem \ref{thm:First-same-exponent} and Corollary \ref{cor:Second-approach-conclusion}, respectively.\\

We briefly comment on the same question when integers are replaced by rational numbers in \cref{sec:Rational-case}. We note that one of two approaches to solving Question \ref{q?-2} still works over rationals, as explained in \S\ref{subsec:rational-numbers-same-exponent}. However, our approach for Question \ref{q?-1} does not extend to rational numbers; but in the meantime, Question \ref{conj:General-rationals} was solved by Santicola \cite{ReverseDE-Q} by clever and precise refining of the arguments presented here, see \S\ref{subsec:rational-numbers-general-case}.

\section{Famous Diophantine problems}\label{sec:FamousDE}
\subsection{Famous Diophantine equations}\label{subsec:FamousDE}
Now we present several famous Diophantine equations or related conjectures. We also mention the time needed to solve these equations. We recall that these equations mentioned below ask for or are related to integer solutions. 
\begin{enumerate}
\item {\it Fermat's Last Theorem}: $x^n+y^n=z^n$, with an integer $n\geq 3$. It took more than 350 years until it was proven in series of papers by Wiles et al. that this equation has only trivial solutions, i.e., such that $xyz=0$.
\item {\it Generalized Fermat's Equation}: $x^k+y^l=z^m$, where $k,l,m\geq2$ are integers, see \cite{Bennett-Mihailescu-Siksek} for more details. A special case of this equation is related to the {\it Beal conjecture} which states that if 
$$k,l,m\geq3,$$ then for any solution $(x,y,z)\in \Z^3$, there is a prime number $p$ such that $p\mid x$, $p\mid y$, $p\mid z$, i.e., solutions cannot be triples of coprime integers. Note that the condition that $k,l,m\geq3$ is necessary, as otherwise one can find coprime solutions, such as $1^k+2^3=(\pm 3)^2$, for any $k\geq 2$, or, for example, $2^5+7^2=3^4$. More identities can be found in \cite{Bennett-Mihailescu-Siksek}. 

For fixed $k,l,m\geq 2$ such that $$1/k+1/l+1/m<1,$$ Darmon and Granville \cite{Darmon-Granville} proved that there are only finitely many triples of coprime integers $(x,y,z)$ such that $x^k+y^l=z^m$. 

If we vary $k,l,m\geq 2$ such that $1/k+1/l+1/m<1$, there is a conjecture, the {\it Fermat–Catalan conjecture}, stating that there are only finitely many sextuples 
\[(x,y,z,k,l,m)\] such that $x^k+y^l=z^m$ and $x,y,z$ are coprime.
\item Former {\it Catalan's conjecture}, now {\it Mih\u{a}ilescu's theorem}: The equation $x^a-y^b=1$ with integers $a,b\geq 2$ and $x,y>0$ has only one solution $3^2-2^3=1$. In other words, the only two positive integers which are consecutive perfect powers of integers are 8 and 9. This statement was conjectured by Catalan and proved by Mih\u{a}ilescu \cite{Mihailescu-Catalan} slightly more than 150 years later.
\item {\it (Generalized) Ramanujan-Nagell equation}: The Ramanujan-Nagell equation is the equation $x^2+7=2^n$, where $x$ and $n$ are integers. Ramanujan conjectured that there are five values of $n$ for which the equation has a solution, $n\in\{3,4,5,7,15\}$, and Nagell proved this 35 years later. It was first published in Norwegian, and later in English \cite{Nagell-Ramanujan-Nagell}. The generalized Ramanujan-Nagell equation is an equation of the shape $x^2+D=y^n$, where $x,n,D\in \Z$ and $n\geq 3$. It is widely studied, for example, see \cite{Bugeaud-Mignotte-Siksek}, or \cite{Maohua-Gokhan-Surve-generalised-Ramanujan-Nagell} for a recent survey on this equation. Note that some authors, such as, in \cite{Moree-Stewart}, consider the slightly different equation $ F(x)=p_1^{e_1}\cdots p_s^{e_s}$, for some $F\in\Z[x]$ and  $p_1, ..., p_s$ prescribed
primes and $e_i\geq 0$ (unknown) exponents.
\end{enumerate}

\subsection{Rational points on curves}\label{subsec:RPoncurves}
One of the most important examples of Diophantine problems is the following trichotomy of rational points on curves defined over $\Q$ (and in fact, over any number field $K/\Q$). Let $C/\Q$ be a {\it nice} curve; here, the adjective {\it nice} is a well-known notation used for a smooth, projective, and geometrically irreducible curve. Denote the genus of $C$ by $g(C)$.

\begin{enumerate}[(a)]
\item\label{g=0} ({\it known for a long time}) If $g(C)=0$, then either $C(\Q)=\emptyset$ (e.g., consider $C:x^2+y^2=3$) or $C(\Q)$ is infinite (e.g., consider $C:x^2+y^2=1$, this curve can be used to parametrize the Pythagorean triples) and furthermore isomorphic to $\mathbb{P}^1(\Q)$.

\item\label{g=1} ({\it Mordell, 1922} \cite{Mordell}) If $g(C)=1$ and $C(\Q)\neq \emptyset$, then $C$ is called an {\it elliptic curve} and there is a group law on $C(\Q)$, which makes $C(\Q)$ a finitely generated group, i.e., $C(\Q)\cong \Z^r\bigoplus T$, where $T$ is a finite (torsion) subgroup, and $r$ is called the {\it rank} of $C$ over $\Q$. 

There are only finitely many possibilities for the torsion subgroup $T$ (exactly 15), as these were classified by Mazur \cite{Mazur-torsion-MCEI}, \cite{Mazur-torsion-RPMC}. To prove this statement, Mazur determined rational points on certain types of curves, called modular curves. 

Computing the rank $r$ is still a difficult problem, and there are some ways that might succeed in computing it, but there is still no guarantee that the known algorithms can compute the rank of all elliptic curves, look, for example, at Silverman's book \cite{SilvermanAEC}. Also, these algorithms for computing ranks are quite complicated, and in practice, it is challenging for humans to perform them. Hence, these algorithms are implemented in computer algebra systems, such as \texttt{Magma} \cite{Magma}.
\item\label{g>=2} ({\it Faltings, 1983} \cite{Faltings}) If $g(C)\geq 2$, then $C(\Q)$ is finite. In \cite{Mordell}, Mordell conjectured this statement, so it took about 60 years until it was proved. This statement was so difficult and significant that Faltings won a Fields medal for this proof. By now, there are a few different proofs of this statement. 

However, there are no practical algorithms to compute $C(\Q)$ for a given curve $C$, and it is a very active area of research to find methods that can compute precise sets of rational points on curves. As we have already seen in the case of elliptic curves, there are significant problems that reduce to computing rational points on curves (e.g., Mazur's theorem), so it is indeed very important nowadays to further develop existing methods for determining rational points on curves. 

Unlike elliptic curves, the set $C(\Q)$ for $g(C)\geq 2$ has no particular algebraic structure, so sometimes we want to study it by embedding it into an object with more structure, called the {\it Jacobian} $J$ of the curve $C$; this is an abelian variety. As for elliptic curves, Weil \cite{Weil} in 1929 proved, now called Mordell-Weil theorem, that $J(\Q)$ is a finitely generated abelian group, i.e., $J(\Q)\cong \Z^r\bigoplus T$, where $T$ is a finite (torsion) subgroup and $r$ is called the {\it rank} of $J$. In contrast to elliptic curves, much less is known about possibilities for $T$, and the computation of $r$ is much more difficult. In some cases, there are ways to do so, for example, for Jacobians of hyperelliptic curves, by Stoll \cite{Stoll-Implementing-2-descent}, which is implemented in \texttt{Magma}.
\end{enumerate}

\section{The same exponent}\label{sec:The-same-exp}

Fix $m\geq 2$. Let $S=\{a_1^m,\ldots,a_k^m\}$ be a set of $k$ distinct $m$th powers of integers $a_1,\ldots,a_k$. We now construct a polynomial $f_S\in \Z[x]$ such that $f_S(\Z)\cap \P_m=S$.

\subsection{First approach}\label{subsec:Runge-approach}
The solution is inspired by Runge's method (invented in 1887 by Runge \cite{Runge-Original-Paper}; see \cite[Chapter 5]{Schoof-Book-Catalan} for a nice survey). We first define an auxiliary polynomial
\[
g(x)=(x-a_1)\cdots(x-a_k).
\]
Consider a polynomial 
\[
f_S(x)=(x(x^2+1)g(x))^{4m}+(x^{2m}-x^2+2)g(x)^{2m}+x^m.   
\]

\begin{theorem}\label{thm:First-same-exponent} 
We have
\begin{enumerate}[(i)]
\item $S\subseteq f_S(\Z)\cap \P_m$;
\item $f(x)\notin \P_m$ for any $x\in \Z\setminus \{0,a_1,\ldots,a_k\}$;
\item $f_S(\Z)\cap \P_m=S$.
\end{enumerate}
\end{theorem}

\begin{proof}\hfill
\begin{enumerate}[(i)]
\item We evaluate (note that this statement is true regardless whether $a_1\cdots a_k=0$ or not)
\[
f_S(0)=2(a_1\cdots a_k)^{2m},\;\; g(a_i)=0,\;\; f_S(a_i)=a_i^m,\;\; \text{for $1\leq i\leq k$}.
\]
\item Let $x\in \Z\setminus \{0,a_1,\ldots,a_k\}$. We first note that $|g(x)|=|(x-a_1)\cdots(x-a_k)|\geq 1$. Then
\[
(x^{2m}-x^2+2)g(x)^{2m}\geq x^{2m}-x^2+2\geq x^{2m}-|x|^m+2 > |x|^m.
\]
It follows that $(x^{2m}-x^2+2)g(x)^{2m}+x^m>0$, so 
\[
f_S(x)>(x(x^2+1)g(x))^{4m}=((x(x^2+1)g(x))^4)^{m}.
\]
Now we prove that 
\begin{equation}\label{eq:polynomial-upper-inequality}
f_S(x)<((x(x^2+1)g(x))^{4}+1)^m,    
\end{equation}
implying that $f(x)$ cannot be an $m$th power because it is nested between two consecutive $m$th powers. To prove inequality \eqref{eq:polynomial-upper-inequality}, it suffices to show
\begin{equation}\label{eq:helpful-upper-inequality}
m(x(x^2+1)g(x))^{4m-4}>(x^{2m}-x^2+2)g(x)^{2m}+x^m.
\end{equation}
The following two inequalities will help us to show \eqref{eq:helpful-upper-inequality}. The first one is
\begin{equation}\label{eq:helpful-helpful-upper-inequality-1}
(x(x^2+1)g(x))^{4m-4}> (x^{2m}-x^2+2)g(x)^{2m}
\end{equation}
which holds because $g(x)^{4m-4}\geq g(x)^{2m}$ and the inequality 
\[
(x(x^2+1))^{4m-4}>x^{2m}-x^2+2
\]
is trivial (note that $x\neq 0$). \\

The second one is clear
\begin{equation}\label{eq:helpful-helpful-upper-inequality-2}
(x(x^2+1)g(x))^{4m-4}\geq (x(x^2+1))^{4m-4}>|x|^{m}\geq x^m.
\end{equation}
Inequality \eqref{eq:helpful-upper-inequality} follows from $m\geq 2$ and inequalities \eqref{eq:helpful-helpful-upper-inequality-1} and \eqref{eq:helpful-helpful-upper-inequality-2}.
\item Follows directly from (i) and (ii).
\end{enumerate}
\end{proof}

\subsection{Second approach}\label{subsec:modular-approach}
If $m\geq 3$, there is another way to construct the required polynomial $f_S\in \Z[x]$. The term $(x^{2m}-x^2+2)((x-a_1)\cdots(x-a_k))^{2m}$ in the previous construction was used to ensure that $f(0)$ in not an $m$th power if $a_1\cdots a_k\neq 0$. One could try a simpler construction 
\[
g_S(x)=((x-a_1)\cdots(x-a_k))^{m}+x^m.
\]
By Fermat's Last Theorem, $g_S(x)$ is not an $m$th power unless $x\in\{0,a_1,\ldots,a_k\}$. However, we want to exclude the possibility for $x=0$ if $a_1\cdots a_k\neq 0$. 
We instead use a more general Fermat's Equation, as suggested by Samir Siksek. Consider
\[
f_S(x)=3((x-a_1)\cdots(x-a_k))^{m}+x^m.
\]

\begin{lemma}\label{L:Only-zero-solution} 
Let $m\geq 3$. If $(x_1,x_2,x_3)\in \Z^3$ satisfy $3x_1^m+x_2^m=x_3^m$, then $x_1=0$.
\end{lemma}

\begin{proof}
We distinguish three cases according to whether $m$ has a prime divisor $p\geq 5$, $m$ is a power of 3, or $m\geq 4$ is a power of 2.
\begin{enumerate}[(i)]
\item Let $p\geq 5$ be a prime number dividing  $m$. By a generalized approach to Fermat's Last Theorem by Kraus \cite{Kraus-p-p-p}, the equation $3x_1^p+x_2^p=x_3^p$ has only trivial  solutions $x_1x_2x_3=0$, which implies that $x_1=0$. See also notes by Siksek \cite{Siksek-Modular-approach-notes} for a nice explanation of the modular method. In this concrete case, see \cite[Theorems 1, 15]{Siksek-Modular-approach-notes} for an argument that a non-trivial solution to $3x_1^p+x_2^p=x_3^p$ corresponds to a certain newform of weight 2 and level 6, which does not exist. Hence, if $3x_1^m+x_2^m=x_3^m$, then $x_1=0$.
\item If $m$ is a power of 3, it suffices to consider $m=3$. The cubic curve 
\[
X\colon 3x_1^3+x_2^3=x_3^3
\]
has genus one and $X(\Q)\neq \emptyset$, hence it is isomorphic over $\Q$ to an elliptic curve. As pointed out in \S\ref{subsec:RPoncurves} \eqref{g=1}, using \texttt{Magma} \cite{Magma}, we prove that the rank of $X$  is zero, and that $X(\Q)$ consists only of the point at infinity, which corresponds precisely to $x_1=0$.
\item If $m\geq 4$ is a power of 2, it suffices to consider $m=4$. Integral solutions to $3x_1^4+x_2^4=x_3^4$ correspond to the rational points on a curve 
\[
X'\colon 3x^4+y^4=1.
\]
One can consider a curve 
\[
X/\Q\colon y^2=1-3x^4.
\]
As we stated in \S\ref{subsec:RPoncurves} \eqref{g>=2}, we can embed $X$ into its Jacobian $J$, and, using \texttt{Magma}, we prove $J(\Q)\simeq \Z/2\Z$. Hence, we conclude $\#X(\Q)\leq \#J(\Q)= 2$, so $X(\Q)=\{(0,\pm1)\}$, implying $X'(\Q)=\{(0,\pm1)\}$, and hence, that $x_1=0$. 
\end{enumerate}
\end{proof}

\begin{corollary}\label{cor:Second-approach-conclusion}
We have $f_S(\Z)\cap \P_m=S$.
\end{corollary}
\begin{proof}
It follows by Lemma \ref{L:Only-zero-solution} that $f_S(x)\in \P_m$ only when $x\in S$, and we compute directly that $f_S(a_i)=a_i^m$ for every $1\leq i\leq k$.
\end{proof}

We note that this approach does not work directly for $m=2$ as the equation
\[qx_1^2+x_2^2=x_3^2\] has solutions with $x_1\neq0$ for any $q\in \Z$. We may assume that $q>0$. If $q$ is not a square of an integer, then, consider $x_2=1$ and note that Pell's equation $x_3^2-qx_1^2=1$ has infinitely many solutions. If $q=r^2$ is a square of an integer $r$, then $(rx_1,x_2,x_3)$ is a Pythagorean triple, and we can find infinitely many of them (for example, let $x_1=2s$, $x_2=s^2-r^2$, $x_3=s^2+r^2$, for $s\in\Z$).

\section{General case}\label{sec:General-case}

In this section, using Mih\u{a}ilescu's theorem, we construct a polynomial with the desired property in the general case. Let $S=\{b_1,\ldots,b_k\}\subseteq\P$ be a finite set of perfect integer powers. We construct a polynomial $f_S\in \Z[x]$ such that $f_S$ is the identity on $S$, and $f_S(x)\notin \P$, for all $x\in \Z\setminus S$. Define an auxiliary polynomial 
\[
g(x)=((x-b_1)\cdots(x-b_k))^4+1,
\]
and let 
\[f_S(x)=g(x)((x-1)g(x)+1).\]\\
We now prove that $f_S$ satisfies the property asked in  Question \ref{q?-1}, hence giving a positive answer to Question \ref{q?-1}. 

\begin{theorem}\label{thm:Second-general}\hfill
\begin{enumerate}[(i)]
\item Let $x\in \Z$. Then $f_S(x)\in \P$ if and only if $x\in S$.
\item We have $f_S(\Z)\cap \P = S$.
\end{enumerate}
\end{theorem}
\begin{proof}\hfill
\begin{enumerate}[(i)]
\item Let $x\in \Z$ be such that $f_S(x)=y^n$, for some integers $y,n$ with $n\geq 2$. Since $g(x)$ and $(x-1)g(x)+1$ are coprime integers, we conclude that there is $z\in \Z$ such that $g(x)=\pm z^n$, and since $g(x)>0$, we may assume that $g(x)=z^n$. Denote 
\[c\colonequals (x-b_1)\cdots(x-b_k).\] Then we have
\[
z^n-c^4=1.
\]
Since the exponent of $c$ is greater than 3, by Mih\u{a}ilescu's theorem, the only possibility is that $z=1$ and $c=0$, implying that $x=b_i$ for some $1\leq i\leq k$. If $x\in S$, we evaluate $f_S(x)=x\in \P$.
\item Follows from (i) and the fact that $f_S$ is the identity on $S$.
\end{enumerate}
\end{proof}

\section{Generalizations.}\label{sec:Rational-case}

It is interesting to see whether our approaches work when we ask the same questions with integers replaced by rational numbers. Denote by $\QQ=\{a^m\colon a \in \Q, m\geq 2\}$ the set of all rational powers, and, for a fixed integer $m\geq 2$, let $\QQ_m=\{a^m\colon a\in \Q\}$ be the set of all $m$th powers of rational numbers. We have naturally two questions. 

\subsection{Rational numbers: The same exponent}\label{subsec:rational-numbers-same-exponent}

\begin{question}\label{q?-3}
Let $S\subseteq \QQ_m$ be a finite subset of $m$th rational powers. Is there a polynomial $f_S\in \Q[x]$ such that $f_S(\Q)\cap \QQ_m = S$?
\end{question}

We first note that for $m\geq 3$ we can give a positive answer to Question \ref{q?-3}. It is clear that the first approach cannot be used because there is no version of Theorem \ref{thm:First-same-exponent} that covers rational numbers. In this approach, we use the property that distinct integers differ by at least 1, which is not true for rational numbers; their absolute difference can be arbitrarily small. However, Lemma \ref{L:Only-zero-solution} remains true for rational numbers, i.e., if we replace the condition $(x_1,x_2,x_3)\in \Z^3$ by $(x_1,x_2,x_3)\in \Q^3$. This follows because the equation $3x_1^m+x_2^m=x_3^m$ is homogeneous, so any rational solution easily leads to an integer solution. 
Hence, if $S=\{a_1^m,\ldots,a_k^m\}$ is a set of $k$ distinct $m$th powers of rational numbers $a_1,\ldots,a_k$, where $m\geq 3$, then again, for the polynomial
\[
f_S(x)=3((x-a_1)\cdots(x-a_k))^{m}+x^m,
\]
we have that $f_S(\Q)\cap \QQ_m = S$.\\

\subsection{Rational numbers: General case}\label{subsec:rational-numbers-general-case}

On the other hand, the approach from \cref{sec:General-case} does not work directly over rational numbers. We cannot use the coprimality argument in the factorization so easily; we would need to take care of possible denominators. Hence, we can formulate the question:

\begin{question}\label{conj:General-rationals}
Let $S\subseteq \QQ$ be a finite set of perfect rational powers. Is there a polynomial $f_S\in \Q[x]$ such that $f_S(\Q)\cap \QQ = S$?
\end{question}

This question was answered affirmatively by Santicola \cite{ReverseDE-Q}, who noted that it is sufficient to use special cases of Mih\u{a}ilescu's theorem proven by Lebesgue \cite{Lebesgue} to answer Question \ref{q?-2}. Furthermore, Santicola's construction uses the results of \cite{Bennett-Ellenberg-Ng, BruinEllCh, EllenbergGFE}. To handle possible denominators and to prove the key lemma \cite[Lemma 6]{ReverseDE-Q}, Santicola used the result of Peth\H{o} \cite{Petho}, and independently of Shorey and Stewart \cite{Shorey-Stewart}.  

\subsection{Challenge}\label{subsec:challenge}

We see that this article has already inspired further research. It is also interesting to consider this question over other rings and fields. We challenge the interested reader to try to answer the analogous questions over $\Z[i]$ or $\Q[i]$, or, in general, over $\mathcal{O}_K$ or $K$, where $K/\Q$ is any number field.

\subsection*{Acknowledgment.}
The author would like to thank Samir Siksek for raising the questions, discussing the topic and giving useful comments and suggestions, Michael Stoll for organizing the conference ``Rational Points 2022'' and for the financial support for the conference, Ludwig Fürst and Himanshu Shukla for conversations about this topic and writing a report on the problem session, and Pieter Moree, Oana Padurariu, and Alain Togbe for many valuable comments that improved the article. The author was supported by a guest postdoc fellowship at the Max Planck Institute for Mathematics in Bonn, by Czech Science Foundation GAČR, grant 21-00420M, and by Charles University Research Centre program UNCE/SCI/022 during various
stages of this project.

\end{document}